\documentclass[11pt]{amsart}
\usepackage{amssymb,amsthm}
\usepackage[nosumlimits]{mathtools}
\usepackage{enumerate}
\usepackage{tabularray}
\usepackage[cal=euler,scr=rsfso]{mathalpha}
\usepackage[margin=1in]{geometry}
\usepackage{stmaryrd}
\usepackage{upgreek}
\usepackage{tikz-cd}

\theoremstyle{plain}
\newtheorem{theorem}{Theorem}[section]
\newtheorem{proposition}[theorem]{Proposition}

\newtheorem{corollary}[theorem]{Corollary}

\theoremstyle{definition}
\newtheorem{definition}[theorem]{Definition}

\theoremstyle{remark}

\usepackage{amsthm}

\usepackage{hyperref}
\usepackage{xcolor}
\hypersetup{
    colorlinks,
    linkcolor={red!50!black},
    citecolor={blue!50!black},
    urlcolor={blue!80!black}
}

\newcommand{\tp}{{\scriptscriptstyle\mathsf{T}}}

\newcommand{\pp}{{\scriptscriptstyle++}}

\let\latexcirc=\circ
\newcommand{\ccirc}{\mathbin{\mathchoice
  {\xcirc\scriptstyle}
  {\xcirc\scriptstyle}
  {\xcirc\scriptscriptstyle}
  {\xcirc\scriptscriptstyle}
}}
\newcommand{\xcirc}[1]{\vcenter{\hbox{$#1\latexcirc$}}}
\let\circ\ccirc

\let\O\undefined
\DeclareMathOperator{\O}{O}

\DeclareMathOperator{\Flag}{Flag}
\DeclareMathOperator{\V}{V}

\DeclareMathOperator{\tr}{tr}
\DeclareMathOperator{\Gr}{Gr}

\DeclareMathOperator{\diag}{diag}
\DeclareMathOperator{\rank}{rank}

\begin{document}
\title{Stiefel optimization is NP-hard}
\author[Z.~Lai]{Zehua~Lai}
\address{Department of Mathematics, University of Texas, Austin, TX 78712}
\email{zehua.lai@austin.utexas.edu}
\author[L.-H.~Lim]{Lek-Heng~Lim}
\address{Computational and Applied Mathematics Initiative, Department of Statistics,
University of Chicago, Chicago, IL 60637}
\email{lekheng@uchicago.edu}
\author[T.~Tang]{Tianyun Tang}
\address{Computational and Applied Mathematics Initiative, Department of Statistics,
University of Chicago, Chicago, IL 60637}
\email{tianyuntang@uchicago.edu}

\begin{abstract}
We show that linearly constrained linear optimization over a Stiefel or Grassmann manifold is NP-hard in general. We show that the same is true for unconstrained quadratic optimization over a Stiefel manifold. We will show that unless $\mathrm{P}=\mathrm{NP}$, these optimization problems over a Stiefel manifold do not have $\mathrm{FPTAS}$. As an aside we extend our results to flag manifolds. Combined with earlier findings, this shows that manifold optimization is a difficult endeavor --- even the simplest problems like LP and unconstrained QP are already NP-hard on the most common manifolds.
\end{abstract}
\maketitle

\section{Introduction}\label{sec:intro}

Aside from the Euclidean $n$-space, the three most common manifolds in applications are the Stiefel manifold of orthonormal $k$-frames in $n$-space, the Grassmann manifold of $k$-planes in $n$-space, and the Cartan manifold of centered ellipsoids in $n$-space. They also constitute the three canonical examples in manifold optimization \cite{EAS99, AMS}, with simple representations as submanifolds of orthonormal, projection, and positive definite matrices respectively:
\begin{equation}\label{eq:models}
\small
\begin{tblr}{@{}lll}
\textsc{manifold}  & \textsc{object} &  \textsc{matrix model} \\\hline
\text{Euclidean} & \text{points}  & \mathbb{R}^n \\
\text{Stiefel}  & k\text{-frames}  & \V(k,n) = \{ X \in \mathbb{R}^{n \times k} : X^\tp X = I \} \\
\text{Grassmann} & k\text{-planes}  & \Gr(k,n) = \{ X \in \mathbb{R}^{n \times n} : X^2 = X = X^\tp, \; \tr(X) = k \} \\
\text{Cartan} & \text{ellipsoids}  & \mathbb{S}^n_\pp = \{ X \in \mathbb{R}^{n \times n} : X = X^\tp, \; X \succ 0 \} \\
\end{tblr}
\end{equation}

The goal of this article is to fill in the gaps left unaddressed in \cite{ZLK24b}, which includes the following findings: \emph{unconstrained} QP over $\Gr(k,n)$ and $\mathbb{S}^n_\pp$ is NP-hard  \cite[Theorem~5.3 and Corollary~8.2]{ZLK24b} and furthermore has no $\mathrm{FPTAS}$ unless $\mathrm{P}=\mathrm{NP}$; on the other hand, \emph{unconstrained} LP over $\Gr(k,n)$, $\V(k,n)$, and $\mathbb{S}^n_\pp$ has closed-form polynomial-time solution \cite[Lemma~9.1]{ZLK24b}.

A glaring omission is unconstrained QP over $\V(k,n)$, which was left as an open problem in \cite[Section~10]{ZLK24b}. One may deduce that unconstrained \emph{cubic} programming is NP-hard over $\V(k,n)$ \cite[Theorem~7.2]{ZLK24b}. It is also well-known that unconstrained QP is NP-hard for $\V(n,n) = \O(n)$  \cite{Nem} and  polynomial-time for $\V(1,n) = \{x \in \mathbb{R}^n : \lVert x \rVert =1 \}$ \cite[Section~4.3]{vava}. But aside from these boundary cases, the computational complexity of unconstrained QP over $\V(k,n)$ for $1 < k < n$ is unknown.

Another omission of \cite{ZLK24b} is \emph{constrained} LP, i.e., optimization of a linear objective under linear constraints. It is well-known that LP is polynomial-time over $\mathbb{R}^n$ \cite{Kha} and the same is essentially true for LP over $\mathbb{S}^n_\pp$ --- semidefinite programming (SDP) solves it to arbitrary accuracy under mild assumptions \cite{NN}. LP has been shown to be NP-hard over $\V(k, n)$ in \cite[Proposition~2.2]{song2024} but the approach therein cannot be directly extended to demonstrate the nonexistence of $\mathrm{FPTAS}$. Nothing is known about the complexity of LP over $\Gr(k,n)$. More generally, $\Gr(k,n)$ is the $p = 1$ special case of a \emph{flag manifold} $\Flag(k_1,\dots,k_p,n)$ \cite{YWL}. For $p > 1$, the complexity of unconstrained QP and constrained LP over $\Flag(k_1,\dots,k_p,n)$ are both open.

Given that any constrained optimization problem likely contains LP as a special or degenerate case and any unconstrained optimization problem likely contains unconstrained QP as a special or degenerate case, it is natural to begin our investigation from these simplest problems. We summarize the state of our current knowledge in the following table:
\begin{center}\small
\begin{tblr}{
  colspec={l|c|c},
}
\SetCell[r=2]{c}{\textsc{manifold}} & \SetCell[c=2]{c}{\textsc{problem complexity}} \\
\hline
 & \textsc{LP} & \textsc{unconstrained QP} \\ \hline
Euclidean & P & P \\
Stiefel & NP-hard & \textbf{?} \\
Grassmann & \textbf{?} & NP-hard \\
Flag & \textbf{?} & \textbf{?} \\
Cartan & SDP & NP-hard
\end{tblr}
\end{center}

In this article, we will fill in the four missing gaps --- we will show that they are all NP-hard and that, unless $\mathrm{P}=\mathrm{NP}$, they have no $\mathrm{FPTAS}$. We will also provide a different reduction from the one used in \cite[Proposition~2.2]{song2024}, showing that LP over $\V(k, n)$ has no $\mathrm{FPTAS}$  unless $\mathrm{P}=\mathrm{NP}$. Note that this is already known for unconstrained QP over  $\Gr(k,n)$ and $\mathbb{S}^n_\pp$  by our earlier work in \cite{ZLK24b}. The established intractability extends to other models of these manifolds: As was shown in \cite{ZLK24b}, all known models of the Stiefel manifold may be transformed to one another in polynomial time, and likewise for all known models of the Grassmannian. 

All our reductions in this article will be from the stability number, max-cut, and clique number of unweighted undirected graphs, all famously NP-hard and, more importantly, are NP-hard to approximate unless $\mathrm{P}=\mathrm{NP}$ \cite{Has,Zuck}.

\subsection{Conventions and background}\label{sec:note}

When we write $\V(k,n)$ or $\Gr(k,n)$ in this article, we refer to the matrix models as described in \eqref{eq:models}.

We will denote entries of a matrix $X \in \mathbb{R}^{m \times n}$ in lower case $x_{ij}$, $i =1,\dots,m$, $j =1,\dots,n$. When delimited in parentheses, $(x_1,\dots,x_n)$ will denote a \emph{column} vector. Any vector $x \in \mathbb{R}^n$ will always be a column vector, i.e., $\mathbb{R}^n = \mathbb{R}^{n \times 1}$.  For any $X \in \mathbb{R}^{m \times n}$, we write
\begin{equation}\label{eq:diag}
\diag(X) \coloneqq \begin{cases}
(x_{11},\dots,x_{mm}) \in \mathbb{R}^m & \text{if } m \le n, \\
 (x_{11},\dots,x_{nn}) \in \mathbb{R}^n & \text{if }  n \le m.
\end{cases}
\end{equation}

For any $m \in \mathbb{N}$, we write $G_m$ for an $m$-vertex undirected graph with vertex set $\{1,\dots,m\}$ and edge set $E \subseteq \{1,\dots,m\} \times \{1,\dots,m\}$. We denote edges by $(i,j)$, adopting the convention that $(i,j)\in E$ if and only if $(j,i)\in E$. In particular, a sum over $E$ sums each edge twice, which leads to slightly neater expressions. We do not allow for self-loops so $(i,i) \notin E$ for any $i\in V$.

For easy reference, we reproduce \cite[Definition~2.5]{de2008} here, adapted for our context. Note that the following definition of $\mathrm{FPTAS}$ is also consistent with common usage in computer science \cite[Definition 9.10]{arora2009computational}.
\begin{definition}[Fully polynomial-time approximation scheme]\label{def:FPTAS}
With respect to a maximization problem over $\mathcal{M}$ and a function class $\mathscr{F}$, an algorithm $\mathscr{A}$ is called a fully polynomial-time approximation scheme or $\mathrm{FPTAS}$ if:
\begin{enumerate}[\upshape (i)]
\item For any instance $f \in \mathscr{F}$ and any $\varepsilon>0$, $\mathscr{A}$ takes the defining parameters of $f$ (e.g., coefficients of $f$ when $f$ is a polynomial), $\varepsilon$, and $\mathcal{M}$ as input and computes an $x_\varepsilon \in \mathcal{M}$ such that $f(x_\varepsilon)$ is a $(1-\varepsilon)$-approximation of $f_{\max}$, i.e.,
\[
f_{\max} - f(x_\varepsilon)  \le  \varepsilon(f_{\max} - f_{\min}),
\]
where $f_{\max}$ and $f_{\min}$ are the supremum and infimum of $f$ respectively.

\item The number of operations required for the computation of $x_\varepsilon$ is bounded by a polynomial in the problem size, and $1/\varepsilon$.
\end{enumerate}
\end{definition}

\iffalse
The above definition of $\mathrm{FPTAS}$ is common to continuous optimization. It is not the standard definition of $\mathrm{FPTAS}$ in the algorithms/computational complexity community.
\fi

\section{LP is NP-hard over Stiefel, Grassmann, and flag manifolds}\label{sec:LP}

The reductions in this section will be based on stability number. Let $m \in \mathbb{N}$ and $G_m$ be as above. Recall that a set $S \subseteq \{1,\dots,m\}$ is said to be stable if $(i, j) \notin E$ for all $i, j \in S$. The size of the largest stable set $\alpha(G_m)$, the stability number of $G_m$, is well-known to be NP-hard \cite{Garey}. It has a formulation \cite[Equation~1.4]{Laurent} as a QP over $\mathbb{R}^n$,
\begin{equation}\label{eq:stab}
\alpha(G_m) = \max_{x \in \mathbb{R}^m} \biggl\{ \sum_{i=1}^m x_{i} : x_i + x_j \le 1  \text{ for all } (i,j) \in E, \; x_i^2 = x_i  \text{ for all }i =1,\dots,m\biggr\}.
\end{equation}
Here we use a neutral letter $m$ to denote the number of vertices as, depending on the circumstance, we may have either $k$ or $n$ playing the role of $m$ in our discussions below.

For any $r \le k \le n$, we will formulate the decision problem ``$\alpha(G_k) \ge r$?''\ as LP feasibility over the Stiefel manifold $\V(k, n)$.
\begin{theorem}[Stiefel LP is NP-hard]\label{thm:Stfeas}
Let $k \le n$ be positive integers and $G_k$ be a $k$-vertex undirected graph.
\begin{enumerate}[\upshape (i)]
\item The maximum value of the LP over $\V(k, n)$,
\begin{equation}\label{eq:stabV}
\begin{aligned}
\operatorname{maximize}  &\quad  x_{11} + \dots + x_{kk}\\
\operatorname{subject~to} &\quad  x_{ij} = 0, \; i \neq j,\\
& \quad  x_{ii} + x_{jj} \le 0, \; (i, j) \in E, \\
& \quad  X \in \V(k, n),
\end{aligned}
\end{equation}
is exactly $2\alpha(G_k) - k$. Unless $\mathrm{P} = \mathrm{NP}$, there is no $\mathrm{FPTAS}$ that is polynomial in $n$ and $k$ for LP over $\V(k, n)$.

\item For any $r \le k$, we have $\alpha(G_k) \geq r$ if and only if
\begin{equation}\label{eq:feas}
\biggl\{X \in \V(k, n): x_{ij} = 0\text{ for }i \neq j, \; x_{ii} + x_{jj} \le 0\text{ for all }(i,j) \in E, \; \sum_{i =1}^k x_{ii} \geq 2r-k \biggr\} \ne \varnothing.
\end{equation}
Consequently the LP feasibility problem over $\V(k,n)$ is NP-hard.
\end{enumerate}
\end{theorem}
\begin{proof}

We will first establish the correspondence between the feasible set in \eqref{eq:stabV} and the set of stable subsets of $G_k$. The condition $x_{ij} = 0$, $i \ne j$, taken together with $X^\tp X = I_k$ implies that $X$ is a diagonal matrix with $x_{ii} = \pm 1$, $i =1,\dots,k$. Consider the set $S$ of indices $i$ with $x_{ii} = 1$. For any $i \ne j$, $i, j \in S$, $x_{ii} + x_{jj} = 2 > 0$, so $(i,j) \notin E$. Hence $S$ is a stable set of $G_k$. Conversely, given any stable set $S$ of $G_k$, define the diagonal matrix $X \in \V(k,n)$ with
\[
x_{ii} =
\begin{cases} +1  &\text{if } i \in S,\\
 -1 &\text{if } i \notin S.
\end{cases}
\]
Then $X$ is clearly feasible for \eqref{eq:stabV}. Furthermore, 
\[
\sum_{i =1}^k x_{ii} = |S| - (k - |S|) = 2|S| - k.
\]
Therefore, the maximum value of \eqref{eq:stabV} is $2\alpha(G_k) - k$. In addition, since the maximum value of \eqref{eq:stabV}  can only take on integer values, if we choose a relative error gap of $\varepsilon = 1/k$, then a $(1 - \varepsilon)$-approximation algorithm finds the stability number exactly. So there is no $\mathrm{FPTAS}$ for \eqref{eq:stabV} unless  $\mathrm{P} = \mathrm{NP}$. Lastly, the feasibility problem for \eqref{eq:stabV} is exactly \eqref{eq:feas}. Since the decision problem for stability number is NP-complete, the LP feasibility problem over $\V(k,n)$ is NP-hard.
\end{proof}

The ``no $\mathrm{FPTAS}$'' conclusion in Theorem~\ref{thm:Stfeas} can, in fact, be strengthened. By \cite{Zuck}, unless $\mathrm{P} = \mathrm{NP}$, the maximum of an LP over $\V(k, n)$ cannot be approximated to within a factor of $k^{1-\varepsilon}$ for any $\varepsilon > 0$ with a polynomial-time algorithm.  The NP-hardness of LP over $\V(k, n)$ in \cite{song2024} was established using a reduction from binary linear programming, a decision problem for which the notion of approximability does not apply. It is unlikely that one may deduce a ``no $\mathrm{FPTAS}$'' conclusion from this approach.

For any $k \le n$, we will formulate the decision problem ``$\alpha(G_n) \ge k$?''\ as LP feasibility over the Grassmannian $\Gr(k, n)$. Note that, in this case, it is no longer straightforward to write down an optimization problem like \eqref{eq:stabV} whose optimum value gives $\alpha(G_n)$.
\begin{theorem}[Grassmannian LP is NP-hard]\label{thm:Grfeas}
Let $k \le n$ be positive integers and $G_n$ be an $n$-vertex undirected graph.  Then $\alpha(G_n) \geq k$ if and only if
\[
\bigl\{X \in \Gr(k, n): x_{ij} = 0\text{ for }i \neq j, \; x_{ii} + x_{jj} \le 1\text{ for all }(i,j) \in E \bigr\} \ne \varnothing.
\]
Consequently the LP feasibility problem over $\Gr(k,n)$ is NP-hard.
\end{theorem}
\begin{proof}
For a projection matrix $X$, $\tr(X) = \rank(X)$, and so $X \in \Gr(k,n)$ with $x_{ij} = 0$, $i \neq j$, must be a diagonal matrix with exactly $k$ ones on the diagonal. If the set above  is nonempty, then for any $X$ in this set, the ones on the diagonal of $X$ yield a stable set of cardinality $k$ following the same argument in the proof of Theorem~\ref{thm:Stfeas}. So the set is nonempty if and only if a stable set of size $k$ exists.
\end{proof}

With an eigenvalue decomposition, the model of Grassmannian as projection matrices in \eqref{eq:models} is easily seen to take an alternate form as
\[
\Gr(k,n) = \biggl\{ Q \begin{bmatrix} I_k & 0 \\ 0 & 0 \end{bmatrix} Q^\tp  \in \mathbb{S}^n : Q \in \O(n) \biggr\}.
\]
This generalizes to  give the quadratic model \cite[Table~2]{ZLK24b},
\[
\Gr_{a,b}(k,n) = \biggl\{ Q \begin{bmatrix} a I_k & 0 \\ 0 & b I_{n-k} \end{bmatrix} Q^\tp  \in \mathbb{S}^n : Q \in \O(n) \biggr\}
\]
for any $a \ne b$, which represents an exhaustive list of all minimal equivariant models of the Grassmannian of $k$-planes in $n$-space \cite{LK24a}. This last statement generalizes to flag manifolds.

Let $a_1,\dots,  a_{p+1} \in \mathbb{R}$ be any $p+1$ distinct real numbers. For $0\eqqcolon k_0 < k_1 < \dots < k_{p+1} \coloneqq n$, the \emph{flag manifold} of nested subspaces $\mathbb{V}_1 \subseteq \dots \subseteq \mathbb{V}_p$ of dimensions $\dim \mathbb{V}_j = k_j $, $j =1,\dots,p$, in $\mathbb{R}^n$  may be modeled as a set of matrices
\begin{equation}\label{eq:flag}
\Flag(k_1,\dots,  k_p, n)  \coloneqq \left\lbrace
Q {\setlength{\arraycolsep}{0pt}
\begin{bmatrix}
a_1 I_{n_1} & 0 & \cdots & 0\\[-1ex]
0 &  a_2 I_{n_2} &  &  \vdots \\[-1ex]
\vdots &  & \ddots &  0 \\
0 & \cdots & 0 &  a_{p+1} I_{n_{p+1}}
\end{bmatrix}}
Q^\tp\in \mathbb{S}^n : Q\in \O(n) 
\right\rbrace
\end{equation}
where $n_j \coloneqq k_j - k_{j-1} \in \mathbb{N}$, $j = 1,\dots,p+1$. Moreover, any minimal equivariant model of a flag manifold must take the form in \eqref{eq:flag} \cite{LK24a}.  Indeed the model of Grassmannian that we have been using above is just the $p = 1$, $a_1 =1$, $a_2 =0$ case.

Essentially the same proof will show that LP feasibility over the flag manifold is NP-hard, generalizing Theorem~\ref{thm:Grfeas} to all $p > 1$. We will pick $a_1,\dots,  a_{p+1} \in \mathbb{R}$ so that
\begin{equation}\label{eq:para}
a_1 > a_2 > \dots > a_p > a_{p+1} = 0
\end{equation}
and
\begin{equation}\label{eq:para2}
a_1 < 2a_p.
\end{equation}
Note that these are model parameters that can be chosen for our convenience, just as we set $a_1 =1$, $a_2 =0$ in the Grassmannian case.
\begin{theorem}[Flag LP is NP-hard]\label{thm:LPF}
Let $0\eqqcolon k_0 < k_1 < \dots < k_{p+1} \coloneqq n$ and $G_n$ be an $n$-vertex undirected graph. Then $\alpha(G_n) \geq k_p$ if and only if
\[
\bigl\{X \in \Flag(k_1,\dots,  k_p, n): x_{ij} = 0\text{ for }i \neq j, \; x_{ii} + x_{jj} \le a_1\text{ for all }(i,j) \in E \bigr\} \ne \varnothing.
\]
Consequently the LP feasibility problem over $\Flag(k_1,\dots,  k_p, n)$ is NP-hard.
\end{theorem}
\begin{proof}
The proof is nearly the same as that of Theorem~\ref{thm:Grfeas}. The constraints $x_{ij} = 0$, $ i \neq j$, imply that $X$ is diagonal and the only possible values for $x_{ii}$'s are $a_1, \dots , a_{p+1}$. By \eqref{eq:para2}, the constraint $x_{ii} + x_{jj} \le a_1 < a_p + a_p$ for $(i,j) \in E$ implies that if $(i,j) \in E$, then at least one of $x_{ii}$ or $x_{jj}$ must be $0$. So the set is nonempty if and only if a stable set of size $k_p$ exists.
\end{proof}
In Proposition~\ref{prop:LPF}, we will see that \emph{unconstrained} LP over the flag manifold is polynomial-time solvable.

\section{Unconstrained QP  is NP-hard over Stiefel manifold}

The reduction in this section will be based on maximum cut.  Let $k \in \mathbb{N}$ and $G_k$ be as in Section~\ref{sec:note}. A cut of a partition of the vertex set $\{1,\dots,m\} = S \cup S^{\mathsf{c}}$ is the number of edges $(i,j) \in E$ with $i \in S$ and $j \in S^{\mathsf{c}}$. The size of the largest cut $\kappa(G_k)$, the max-cut of $G_k$, is again a celebrated NP-hard problem \cite{Garey}.  Let $A\in \mathbb{S}^k$ be the adjacency matrix of $G_k$. Then max-cut may be determined from the following QP with $\pm 1$-valued variables:
\begin{equation}\label{eq:maxcut}
4 \kappa(G_k)-2|E|+k =  \max_{x\in \{-1,1\}^k} x^\tp (I_k-A) x.
\end{equation}
This is a slight reformulation of \cite[p.~1119]{GW} that conforms to our convention on graphs in Section~\ref{sec:note}. In the following $\diag(X) \in \mathbb{R}^k$ as defined  in \eqref{eq:diag}.
\begin{theorem}[Unconstrained Stiefel QP is NP-hard]\label{thm:StQP}
Let $k \le n$ be positive integers and $G_k$ be a $k$-vertex undirected graph  with adjacency matrix $A\in \mathbb{S}^k$. The maximum of the unconstrained QP over $\V(k,n)$,
\begin{equation}\label{eq:StQP}
\max_{X\in \V(k,n)} \diag(X)^\tp  (I_k-A) \diag(X),
\end{equation}
is exactly $4 \kappa(G_k)-2|E|+k$. 
Unless $\mathrm{P} = \mathrm{NP}$, there is no $\mathrm{FPTAS}$ that is polynomial in $n$ and $k$ for unconstrained QP over $\V(k, n)$.
\end{theorem}
\begin{proof}
We first show that \eqref{eq:maxcut} is equivalent to the following box-constrained QP problem:
\begin{equation}\label{eq:box}
\max_{x\in [-1,1]^k} x^\tp (I_k-A) x .
\end{equation}
Note that this quadratic form is nonconvex and so the equivalence does not follow from Bauer maximum principle; and while there are similar formulations  \cite[Equation~4]{de2008}, we found none like \eqref{eq:box} that perfectly suits our need here. So we will provide a proof of this equivalence for convenience. The adjacency matrix $A$ has all diagonal entries zero; so the diagonal entries of $I_k - A$ are strictly positive. This is the only observation needed to establish the equivalence. Let $x_*=(x_1^*,\dots,x_n^*)\in [-1,1]^k$ be a maximizer of \eqref{eq:box}. We want to show that $x_*\in \{-1,1\}^k$. Suppose on the contrary that there exists $i\in \{1,\dots,k\}$ with $-1<x^*_i<1$. For any $t\in \mathbb{R}$, consider the vector $x(t) \coloneqq x_*+t e_i$, where $e_i$ is the $i$th column of $I_k$. As $-1<x^*_i<1$,  $x(t)\in [-1,1]^k$ when $|t|$ is sufficiently small.  Moreover,
\begin{align*}
x(t)^\tp (I_k-A) x(t) &=x_*^\tp (I_k-A) x_*+2 t e_i^\tp (I_k-A) x_*+t^2 e_i^\tp (I_k-A) e_i \\
&=x_*^\tp (I_k-A) x_*+\beta t+t^2
\end{align*}
for some $\beta\in \mathbb{R}$. So for sufficiently small $t$, we have $x(t)^\tp (I_k-A) x(t)>x_*^\tp (I_k-A) x_*$, contradicting the optimality of $x_*$ for \eqref{eq:box}.

Next we show that \eqref{eq:StQP} is also equivalent to \eqref{eq:maxcut} and \eqref{eq:box}. Let $f_1$ and $f_2$ be the maximal values of \eqref{eq:StQP} and \eqref{eq:box} respectively. For any $X\in \V(k,n)$, we have $\diag(X)\in [-1,1]^n$.  So $f_1\leq f_2$.  If $x_* = (x_1^*,\dots,x_n^*) \in [-1,1]^n$ is a maximizer of \eqref{eq:box}, then $x_*\in \{-1,1\}^k$ and thus
\[
X_*\coloneqq \begin{bmatrix} x_1^*&  0 & \cdots & 0 \\
0 & x_2^*&   \cdots & 0 \\
\vdots & & \ddots & \vdots \\
0 & 0 & \cdots & x_k^*\\
0 & 0& \cdots & 0\\
\vdots & \vdots & & \vdots \\
0 & 0 & \cdots & 0
\end{bmatrix}\in \V(k,n).
\]
Hence
\[
f_2=x_*^\tp(I_k-A) x_*=\diag(X_*)^\tp  (I_k-A) \diag(X_*)\leq f_1
\]
and we have $f_1=f_2$. The same argument at the end of the proof of Theorem~\ref{thm:Stfeas} shows that with  a relative error gap of $\varepsilon = 1/(2k^2)$, there is no polynomial-time $(1 - \varepsilon)$-approximation algorithm for \eqref{eq:StQP} unless $\mathrm{P} = \mathrm{NP}$.
\end{proof}
 
Again, the ``no $\mathrm{FPTAS}$'' conclusion in Theorem~\ref{thm:StQP} can  be strengthened. By \cite{Has}, unless $\mathrm{P} = \mathrm{NP}$, the maximum of an unconstrained QP over $\V(k, n)$ cannot be approximated to within a factor of $\frac{17}{16}-\varepsilon$ for any $\varepsilon > 0$ with a polynomial-time algorithm.

\section{Unconstrained QP  is NP-hard over flag manifold}

The reduction in this section will be based on clique number. Let $n \in \mathbb{N}$ and $G_n$ be as in Section~\ref{sec:note}. Recall that a set $S \subseteq \{1,\dots,n\}$ is a clique if $(i,j) \in E$ for all $i, j \in S$. The size of the largest clique $\omega(G_n)$, the clique number of $G_n$, is again famously NP-hard \cite{Garey}. It has an equally famous formulation \cite{motzkin1965}  as a QP over the unit simplex,
\begin{equation}\label{eq:MS}
1-\frac{1}{\omega(G_n)} = \max_{x \in \Delta_{1, n}}   \sum_{(i,j) \in E} x_i x_j.
\end{equation}
We denote our unit simplex by $\Delta_{1, n} \coloneqq \{x \in \mathbb{R}^n :  x_1 + \dots + x_n = 1, \; x_i \ge 0, \; i=1, \dots, n\}$ for consistency with later notation. Our convention of summing over each undirected edge twice gives a slightly different expression from Motzkin--Straus's.

We will add to the notations introduced in Section~\ref{sec:LP} for flag manifolds. First we define the vector
\begin{equation}\label{eq:vec}
 (\overbrace{a_1, \dots a_1}^{n_1}, \overbrace{a_2, \dots, a_2}^{n_2}, \dots, \overbrace{a_{p+1}, \dots, a_{p+1}}^{n_{p+1}}) \in \mathbb{R}^n,
\end{equation}
noting that it is the diagonal of the block diagonal matrix appearing in \eqref{eq:flag}. We also observe that any $X \in \Flag(k_1,\dots,  k_p, n)$ has constant trace given by
\begin{equation}\label{eq:bn}
\tr(X) = \sum_{j=1}^{p+1} n_j a_j \eqqcolon b_n.
\end{equation}
Next we define the following partial sums of the vector in \eqref{eq:vec}:
\[
b_1 \coloneqq a_1,\; b_2 \coloneqq 2a_1, \dots, b_{n_1} \coloneqq n_1 a_1, \; b_{n_1+1} \coloneqq n_1 a_1 + a_2, \dots, b_{n_1 + n_2} \coloneqq n_1 a_1 + n_2 a_2, \dots,
\]
noting that $b_{n_1+\dots+n_{p+1}}=b_n$ is exactly \eqref{eq:bn}. We are now ready to establish an extension of \cite[Propositions~5.1 and 5.2]{ZLK24b} to flag manifolds.

\begin{proposition}[Clique number as QP over flag manifold]\label{prop:CQPF}
Let $0 < k_1 < \dots < k_p < n$ and $G_n$ be an $n$-vertex undirected graph. Let
\begin{equation}\label{eq:k}
k \coloneqq \inf\Bigl\{m \in \mathbb{N}: \frac{j}{m} \le \frac{b_j}{b_n} \text{ for }  j= 1, \dots, m \Bigr\}.
\end{equation}
If $\omega(G_n) > k$, then
\begin{equation}\label{eq:clique}
\max_{X \in \Flag(k_1,\dots,  k_p, n)} \sum_{(i,j) \in E} x_{ii} x_{jj} = b_n^2\Bigl(1-\frac{1}{\omega(G_n)}\Bigr).
\end{equation}
\end{proposition}
\begin{proof}
Let $\Delta_{k_1,\dots,k_p,n}$ denote the convex hull of all $n!$ permutations of the vector in \eqref{eq:vec}. Indeed, by the Schur--Horn Theorem \cite{horn1954}, $\Delta_{k_1,\dots,k_p,n}$ is exactly the image of $\Flag(k_1,\dots,  k_p, n)$ under the diagonal map $\diag : \mathbb{R}^{n \times n} \to \mathbb{R}^n$. By  restricting $\diag$ to $\Flag(k_1,\dots,  k_p, n)$, we have a surjection
\begin{equation}\label{eq:sur}
\diag : \Flag(k_1,\dots,  k_p, n) \to  \Delta_{k_1,\dots,k_p,n}, \quad X \mapsto \diag(X).
\end{equation}
Denote the objective in \eqref{eq:clique} by
\[
g : \Flag(k_1,\dots,  k_p, n) \to \mathbb{R}, \quad g(X) = \sum_{(i,j) \in E} x_{ii} x_{jj}.
\]
Note that $g$ depends only on the diagonal entries of $X$ and indeed if we define
\[
f: \Delta_{k_1,\dots,k_p,n} \to  \mathbb{R}, \quad f(x) \coloneqq  \sum_{(i,j) \in E} x_i x_j,
\]
then $g = f \circ \diag $, where we have written $x_i \coloneqq x_{ii}$ to avoid clutter. As \eqref{eq:sur} is a surjection, \eqref{eq:clique} is equivalent to
\[
\max_{x \in \Delta_{k_1,\dots,k_p,n}} f(x) = b_n^2\Bigl(1-\frac{1}{\omega(G_n)}\Bigr).
\]
Now observe that $\Delta_{k_1,\dots,k_p,n}$ is contained in the simplex
\[
\{x \in \mathbb{R}^n: x_1 + \dots + x_n = b_n, \; x_i \geq 0,\; i =1,\dots,n\}.
\]
Thus, by \eqref{eq:MS}, $f$ has an upper bound given by
\[
\max_{x \in \Delta_{k_1,\dots,k_p,n}} f(x) \le  b_n^2\Bigl(1-\frac{1}{\omega(G_n)}\Bigr).
\]
Without loss of generality, we may suppose that $S = \{1, \dots, \omega(G_n)\} \subseteq \{1,\dots,n\}$ is a largest clique. Let $x_* \in \mathbb{R}^n$ be given by coordinates
\[
x_1^* = \dots = x_{\omega(G_n)}^* = \frac{b_n}{\omega(G_n)}, \quad x_{\omega(G_n) + 1}^* = \dots = x_n^* = 0.
\]
Then $x_* \in \Delta_{k_1,\dots,k_p,n}$ by our choice of $k$ in \eqref{eq:k} and the assumption that $\omega(G_n) > k$.  It is easy to see that $f(x_*) =  b_n^2\bigl(1-1/\omega(G_n)\bigr)$ attains the upper bound. 
\end{proof}
The result above is independent of the choice of model parameters $a_1,\dots,a_{p+1} \in \mathbb{R}$ so long as they are distinct. However, for the next result, we will need to assume that they are chosen according to \eqref{eq:para}, in particular, $a_{p+1} = 0$.

We will now extend \cite[Theorem~5.3]{ZLK24b}, which shows that unconstrained QP over $\Gr(k,n)$ is NP-hard, to any $\Flag(k_1,\dots,  k_p, n)$, noting that when $p =1$, $\Flag(k, n) =\Gr(k, n)$. Unlike its LP counterpart in Theorem~\ref{thm:LPF}, we may fix $k_1 < \dots < k_p$ in the following result, requiring only $n$ to grow.
\begin{corollary}[Unconstrained flag QP is NP-hard]
Let $n \in \mathbb{N}$ be arbitrary. Let $0 < k_1 < \dots < k_p$ be $p$ fixed positive integers. Then unless $\mathrm{P} = \mathrm{NP}$, there is no $\mathrm{FPTAS}$ that is polynomial in $n$ for unconstrained QP over $\Flag(k_1,\dots,  k_p, n)$.
\end{corollary}
\begin{proof}
Since we now choose model parameters according to \eqref{eq:para}, we have in particular that $a_{p+1} = 0$. Hence $b_n$ in \eqref{eq:bn} and therefore $k$  in  \eqref{eq:k} are both independent of $n$. Thus, even as $n \to \infty$, we can check all subgraphs of $G_n$ of size $\le k$ to determine if $\omega(G_n) \leq k$ in polynomial time. If there is an $\mathrm{FPTAS}$ polynomial in $n$ for unconstrained QP over $\Flag(k_1,\dots,  k_p, n)$, then we can determine $\omega(G_n)$ in polynomial time.
\end{proof}

What about \emph{unconstrained} LP over the flag manifold? Note that the NP-hardness in Theorem~\ref{thm:LPF} is for \emph{constrained} LP. A variation of \cite[Lemma~9.1(ii)]{ZLK24b} gives us the answer, which we record below for completeness. 
In the following, we will assume that  $a_1,\dots,  a_{p+1} \in \mathbb{R}$ are arranged in descending order like in \eqref{eq:para} but without requiring that $a_{p+1} = 0$.
\begin{proposition}[Unconstrained flag LP]\label{prop:LPF}
For any $A \in \mathbb{R}^{n\times n}$,
\[
\max_{X \in \Flag(k_1,\dots,  k_p, n)} \tr(A^\tp X) = \sum_{i=1}^n \sum_{j=1}^{p+1} \lambda_i n_j a_j 
\]
is attained at $X = Q \diag(a_1 I_{n_1}, a_2 I_{n_2},\dots,a_{p+1} I_{n_{p+1}} )Q^\tp$ where $(A+A^\tp)/2 = Q \Lambda Q^\tp$ is an eigenvalue decomposition with $Q\in \O(n)$, $\Lambda = \diag(\lambda_1, \dots, \lambda_n) \in \mathbb{R}^{n\times n}$, $\lambda_1 \ge  \dots \ge  \lambda_n$.  
\end{proposition}
\begin{proof}
The problem transforms into
\begin{align*}
\max_{X \in \Flag(k_1,\dots,  k_p, n)} \tr(A^\tp X) &= \max_{X \in \Flag(k_1,\dots,  k_p, n)} \tr\biggl(\biggl[\frac{A+A^\tp}{2}\biggr]^\tp X \biggr) 
= \max_{X \in \Flag(k_1,\dots,  k_p, n)} \tr(\Lambda^\tp X)\\
&= \max_{X \in \Flag(k_1,\dots,  k_p, n)} \sum_{i=1}^n \lambda_i x_{ii}
= \max_{x \in \Delta_{k_1,\dots,  k_p, n}} \sum_{i = 1}^n \lambda_i x_i.
\end{align*}
Since this last maximization is just standard LP over $\mathbb{R}^n$, the maximum is attained at the vertices of $\Delta_{k_1,\dots,  k_p, n}$, i.e., the maximizer $x_*$ is a permutation of \eqref{eq:vec}. By the rearrangement inequality, the maximum is attained when $x_1^* \ge x_2^* \ge \dots \ge x_n^*$, i.e., when $X = Q \diag(a_1 I_{n_1}, a_2 I_{n_2},\dots,a_{p+1} I_{n_{p+1}} )Q^\tp$.
\end{proof}

\section{Conclusion}

With these results, we may now update our earlier table of summary to:
\begin{center}\small
\begin{tblr}{
  colspec={l|c|c},
}
\SetCell[r=2]{c}{\textsc{manifold}} & \SetCell[c=2]{c}{\textsc{problem complexity}} \\
\hline
 & \textsc{LP} & \textsc{unconstrained QP} \\ \hline
Euclidean & P & P \\
Stiefel & NP-hard & NP-hard \\
Grassmann & NP-hard & NP-hard \\
Flag & NP-hard & NP-hard \\
Cartan & SDP & NP-hard
\end{tblr}
\end{center}

Proposition~\ref{prop:LPF} and \cite[Lemma~9.1]{ZLK24b} collectively show that unconstrained LP over these manifolds have simple closed-form solutions. Apart from this trivial case, LP and unconstrained QP are, as we pointed out earlier, the simplest possible optimization problems over any manifold. Any constrained optimization problem likely contains LP as a special or degenerate case; and any unconstrained optimization problem likely contains unconstrained QP as a special or degenerate case. The revelation that these simple cases are NP-hard suggests that other more complex optimization problems involving more complex objectives or constraints are almost certainly also NP-hard. We take this as a sign that manifold optimization ought to be approached in a manner similar to polynomial optimization, where tractable convex relaxations play an indispensable role.

\section*{Data availability statement}

This manuscript has no associated data.

\section*{Ethics declarations}

\subsection*{Conflict of interest} The authors declare that they have no conflict of interest.

\bibliographystyle{abbrv}

\begin{thebibliography}{10}

\bibitem{AMS}
P.-A. Absil, R.~Mahony, and R.~Sepulchre.
\newblock {\em Optimization algorithms on matrix manifolds}.
\newblock Princeton University Press, Princeton, NJ, 2008.
\newblock With a foreword by Paul Van Dooren.

\bibitem{arora2009computational}
S.~Arora and B.~Barak.
\newblock {\em Computational complexity}.
\newblock Cambridge University Press, Cambridge, 2009.
\newblock A modern approach.

\bibitem{de2008}
E.~de~Klerk.
\newblock The complexity of optimizing over a simplex, hypercube or sphere: a
  short survey.
\newblock {\em CEJOR Cent. Eur. J. Oper. Res.}, 16(2):111--125, 2008.

\bibitem{EAS99}
A.~Edelman, T.~A. Arias, and S.~T. Smith.
\newblock The geometry of algorithms with orthogonality constraints.
\newblock {\em SIAM J. Matrix Anal. Appl.}, 20(2):303--353, 1999.

\bibitem{Garey}
M.~R. Garey and D.~S. Johnson.
\newblock {\em Computers and intractability}.
\newblock A Series of Books in the Mathematical Sciences. W. H. Freeman and
  Co., San Francisco, CA, 1979.
\newblock A guide to the theory of NP-completeness.

\bibitem{GW}
M.~X. Goemans and D.~P. Williamson.
\newblock Improved approximation algorithms for maximum cut and satisfiability
  problems using semidefinite programming.
\newblock {\em J. Assoc. Comput. Mach.}, 42(6):1115--1145, 1995.

\bibitem{Has}
J.~H{\aa}stad.
\newblock Some optimal inapproximability results.
\newblock {\em J. Assoc. Comput. Mach.}, 48(4):798--859, 2001.

\bibitem{horn1954}
A.~Horn.
\newblock Doubly stochastic matrices and the diagonal of a rotation matrix.
\newblock {\em Amer. J. Math.}, 76:620--630, 1954.

\bibitem{Kha}
L.~G. Khachiyan.
\newblock A polynomial algorithm in linear programming.
\newblock {\em Dokl. Akad. Nauk SSSR}, 244(5):1093--1096, 1979.

\bibitem{ZLK24b}
Z.~Lai, L.-H. Lim, and K.~Ye.
\newblock Grassmannian optimization is {NP}-hard.
\newblock {\em SIAM J. Optim.}, 35(3):1939--1962, 2025.

\bibitem{Laurent}
M.~Laurent.
\newblock Sums of squares, moment matrices and optimization over polynomials.
\newblock In {\em Emerging applications of algebraic geometry}, volume 149 of
  {\em IMA Vol. Math. Appl.}, pages 157--270. Springer, New York, 2009.

\bibitem{LK24a}
L.-H. Lim and K.~Ye.
\newblock Simple matrix models for the flag, {G}rassmann, and {S}tiefel
  manifolds.
\newblock {\em arXiv:2407.13482}, 2024.

\bibitem{motzkin1965}
T.~S. Motzkin and E.~G. Straus.
\newblock Maxima for graphs and a new proof of a theorem of {T}ur\'{a}n.
\newblock {\em Canadian J. Math.}, 17:533--540, 1965.

\bibitem{Nem}
A.~Nemirovski.
\newblock Sums of random symmetric matrices and quadratic optimization under
  orthogonality constraints.
\newblock {\em Math. Program.}, 109(2-3):283--317, 2007.

\bibitem{NN}
Y.~Nesterov and A.~Nemirovskii.
\newblock {\em Interior-point polynomial algorithms in convex programming},
  volume~13 of {\em SIAM Studies in Applied Mathematics}.
\newblock Society for Industrial and Applied Mathematics (SIAM), Philadelphia,
  PA, 1994.

\bibitem{song2024}
M.~Song and Y.~Xia.
\newblock Linear programming on the {S}tiefel manifold.
\newblock {\em SIAM J. Optim.}, 34(1):718--741, 2024.

\bibitem{vava}
S.~A. Vavasis.
\newblock {\em Nonlinear optimization}, volume~8 of {\em International Series
  of Monographs on Computer Science}.
\newblock The Clarendon Press, Oxford University Press, New York, NY, 1991.

\bibitem{YWL}
K.~Ye, K.~S.-W. Wong, and L.-H. Lim.
\newblock Optimization on flag manifolds.
\newblock {\em Math. Program.}, 194(1-2):621--660, 2022.

\bibitem{Zuck}
D.~Zuckerman.
\newblock Linear degree extractors and the inapproximability of max clique and
  chromatic number.
\newblock {\em Theory Comput.}, 3:103--128, 2007.

\end{thebibliography}

\end{document}